\documentclass[12pt]{amsart}
\usepackage{a4wide,enumerate,xcolor}
\usepackage{amsmath,comment}
\allowdisplaybreaks

\let\pa\partial
\let\na\nabla
\let\eps\varepsilon

\newcommand{\R}{{\mathbb R}}
\newcommand{\diver}{\operatorname{div}}

\newtheorem{theorem}{Theorem}
\newtheorem{lemma}[theorem]{Lemma}
\newtheorem{proposition}[theorem]{Proposition}
\newtheorem{remark}[theorem]{Remark}

\begin{document}

\title[Cross-diffusion limits in multispecies kinetic models]{
Cross-diffusion limits in multispecies kinetic models}

\author[A. J\"ungel]{Ansgar J\"ungel}
\address{Institute of Analysis and Scientific Computing, TU Wien, Wiedner Hauptstra\ss e 8--10, 1040 Wien, Austria}
\email{juengel@tuwien.ac.at} 

\author[A. Pollino]{Annamaria Pollino}
\address{Institute of Analysis and Scientific Computing, TU Wien, Wiedner Hauptstra\ss e 8--10, 1040 Wien, Austria}
\email{annamaria.pollino@tuwien.ac.at}

\author[S. Taguchi]{Satoshi Taguchi}
\address{Department of Advanced Mathematical Sciences, Graduate School of Informatics, Kyoto University, Kyoto 606-8501, Japan}
\email{taguchi.satoshi.5a@kyoto-u.ac.jp} 

\date{\today}

\thanks{The first author acknowledges partial support from the Austrian Science Fund (FWF), grant 10.55776/F65, and from the Austrian Federal Ministry for Women, Science and Research and implemented by \"OAD, project MultHeFlo. This work has received funding from the European 
Research Council (ERC) under the European Union's Horizon 2020 research and innovation programme, ERC Advanced Grant NEUROMORPH, no.~101018153. For open-access purposes, the author has applied a CC BY public copyright license to any author-accepted manuscript version arising from this submission.}

\begin{abstract}
Cross-diffusion systems are formally derived from multispecies kinetic models in the diffusion limit. The first limit in the multispecies BGK model of Gross and Krook leads to a variant of the non-isothermal Maxwell--Stefan equations. The second limit in a BGK model with Brinkman-type force term yields generalized Busenberg--Travis equations, which reduce for constant temperature to the classical Busenberg--Travis system for segregating population species. Entropy equalities are derived for the kinetic and cross-diffusion equations.
\end{abstract}

 

\maketitle


\section{Introduction}

Cross-diffusion equations appear in many applications from physics, biology, chemistry, and other fields. They can be derived from random walks on a lattice \cite{ZaJu17}, interacting particle systems \cite{CDJ19}, or fluid approximations \cite{CCDJ24}. In this paper, we study the derivation from kinetic equations. For instance, the diffusion limit in the multispecies Boltzmann model leads to the Maxwell--Stefan equations for gas mixtures \cite{BGS15,BrGr23}. We explore to what extend other cross-diffusion systems can be derived from multispecies BGK (Bhatnagar--Gross--Krook) models \cite{BGK54}, and we determine the entropy equalities for the limiting cross-diffusion models. We investigate two models.

\subsection{First model}

We consider the multispecies BGK model of Gross and Krook \cite{GrKr56}. Let $f_i^\eps(x,\xi,t)$ be the distribution function of the $i$th component of a gas mixture, depending on the position $x\in\R^3$, velocity $\xi\in\R^3$, and time $t\ge 0$ and solving the scaled equations
\begin{align}\label{1.mbgk}
  \eps\pa_t f_i^\eps + \xi\cdot\na_x f_i^\eps = \frac{1}{\eps}
  \sum_{j=1}^N\nu_{ij}(M_{ij}^\eps-f_i^\eps),\quad i=1,\ldots,N,
\end{align}
for $(x,\xi,t)\in\R^3\times\R^3\times(0,\infty)$, where $\eps>0$ relates to the Mach and Knudsen number, $\nu_{ij}>0$ is the collision frequency, and $M_{ij}^\eps$ denotes the equilibrium distribution (Maxwellian) associated to the $i$th and $j$th component. We refer to Appendix \ref{sec.scale} for details on the scaling. The model satisfies the conservation laws of mass, momentum, and energy, the $H$-theorem, and the uniqueness of the equilibrium solution \cite{HHM17}; see Section \ref{sec.bgkmodel} for details. We show that the Chapman--Enskog expansion $f_i^\eps=M_i^0+\eps g_i^\eps$, where $M_i^0=\lim_{\eps\to 0}M_{ii}^\eps$ and $g_i^\eps$ is a correction, leads in the formal limit $\eps\to 0$ to the non-isothermal Maxwell--Stefan equations (see Theorem \ref{thm.bgk})
\begin{align}\label{1.MS1}
  & \pa_t n_i + \diver_x(n_iv_i) = 0, \quad
  \na_x(n_i\theta) = -\sum_{j=1}^N D_{ij}n_in_j(v_i-v_j), \\
  & \frac32\sum_{i=1}^N \pa_t(n_i\theta) 
  + \frac{5}{2}\sum_{i=1}^N\diver_x
  \bigg\{\frac{n_i\theta\bar{v}_i}{\nu_i}
  - \na_x\bigg(\frac{n_i\theta^2}{m_i\nu_i}\bigg)\bigg\} = 0, \quad\mbox{in }\R^3,\ t>0, \label{1.MS2}
\end{align}
for $i=1,\ldots,N$, where $n_i$ is the (limiting) number density, $m_i$ the particle mass of the $i$th species, $v_i$ the velocity of the $i$th species, and $\theta$ the gas temperature. The Maxwell--Stefan coefficients $D_{ij}$ depend on $\nu_{ij}$ and are symmetric, $\nu_i=\sum_{i=1}^N\nu_{ij}$, and $\bar{v}_i$ is a mean value of $(v_1,\ldots,v_N)$. For constant temperature, equations \eqref{1.MS1} were suggested in 1866 by James Max\-well \cite{Max66} for dilute gases and in 1871 by Joseph Stefan \cite{Ste71} for fluids. In contrast to Fick's law, which predicts a linear dependence between the gradient $\na_x n_i$ and the flux $n_iv_i$, the flux in the Maxwell--Stefan approach also depends on the other gradients $\na_x n_j$ for $j\neq i$.

The isothermal Maxwell--Stefan equations were derived from the Boltzmann system by Boudin, Grec, and Salvarani \cite{BGS15} by assuming that the distribution function $f_i^\eps$ is a shifted Maxwellian. Anwasia and Simi\'c \cite{AnSi22} justified this ansatz by showing that the shifted Maxwellian is the formal minimizer of the kinetic entropy subject to the moment constraints. Briant and Grec \cite{BrGr23} proved rigorously the asymptotic limit $\eps\to 0$ in the linearized Boltzmann system, leading to the Maxwell--Stefan equations in the Fick--Onsager formulation, where the flux is written as a linear combination of the density gradients. The existence of global weak solutions to the Maxwell--Stefan system was proved in \cite{JuSt13}. The Maxwell--Stefan equations coupled to a heat equation for the mixture temperature were derived and analyzed by Hutridurga and Salvarani \cite{HuSa17,HuSa18}. Another non-isothermal Maxwell--Stefan model was determined by Grec and Simi\'c \cite{GrSi23}, using a pressure tensor instead of multispecies temperatures. 

Our model \eqref{1.MS1}--\eqref{1.MS2} slightly differs from other non-isothermal Maxwell--Stefan equations. The works \cite{AnSi22,HuSa17} include the convective term $\diver_x(n_i\theta v_i)$ in the energy equation but no diffusion term. Moreover, the convective velocity in our model is given by the average $\bar{v}_i$ instead of $v_i$. While the diffusion term in \eqref{1.MS2} differs from conventional expressions, like in \cite{GeJu24}, the model in \cite{HeJu21} leads to the right-hand side of \eqref{1.MS2} under a specific choice of the heat conductivity. 

\subsection{Second model}

We perform a Chapman--Enskog expansion for the simple BGK model
\begin{align*}
  \frac{\eps}{\sigma}\pa_t f_i^\eps 
  + \xi\cdot\na_x f_i^\eps
  + \frac{1}{\sigma}\na_x\phi_i^\eps\cdot\na_\xi f_i^\eps
  = \frac{\nu_i}{\eps}(M_i(f_i^\eps)-f_i^\eps),
  \quad i=1,\ldots,N,
\end{align*}
where $\phi_i^\eps$ solves 
\begin{align}\label{1.brink}
  -\eps\Delta_x\phi_i^\eps + \phi_i^\eps 
  = -\sum_{j=1}^N a_{ij}\rho_j^\eps\theta_j^\eps
  \quad\mbox{in }\R^3,\ i=1,\ldots,N.
\end{align}
Here, $M_i(f_i)$ is a Maxwellian having the same moments as $f_i^\eps$ (up to second order) and $\nu_i$ is a collision frequency (not to be confused with $\nu_i$ from the first model). Furthermore,
\begin{align*}
  \rho_i^\eps = m_in_i^\eps = m_i\int_{\R^3}f_i^\eps d\xi, \quad
  \theta_i^\eps = \frac{m_i}{3n_i^\eps}\int_{\R^3}f_i^\eps|\xi|^2 d\xi
\end{align*}
are the mass density and temperature of the $i$th species, respectively, and $a_{ij}\in\R$ are interaction coefficients. Here, we interpret $f_i^\eps$ as the distribution function of active agents or animals of different species with masses $m_1,\ldots,m_n$. The right-hand side of \eqref{1.brink} represents (up to the sign) the mixture pressure for an ideal gas. Accordingly, \eqref{1.brink} can be interpreted as the Brinkman equation that is an extension of the Darcy law \cite{Bri49}. The unconventional point here is that the Brinkman potential $\phi_i^\eps$ depends on the distribution function in a nonlocal way. 

The diffusion limit leads to the equations (see Theorem \ref{thm.BT})
\begin{align*}
  & \pa_t\rho_i + \diver J_i = 0, \quad
  J_i = -\frac{1}{\nu_i}\bigg\{\sigma
  \na_x\bigg(\frac{\rho_i\theta_i}{m_i}\bigg) 
  - \rho_i\na_x\phi_i\bigg\}, \quad
  \phi_i = -\sum_{j=1}^N a_{ij}\rho_j\theta_j, \\
  & \frac32\pa_t\bigg(\frac{\rho_i\theta_i}{m_i}\bigg) 
  = \frac52\diver_x\bigg\{\frac{\sigma}{\nu_i}\na_x
  \bigg(\frac{\rho_i\theta_i^2}{m_i^2}\bigg)
  - \frac{\rho_i\theta_i}{\nu_i m_i}\na_x\phi_i\bigg\}
  + \frac{1}{\sigma}J_i\cdot\na_x\phi_i, \quad i=1,\ldots,N.
\end{align*}
The last term in the energy equation is the Joule heating with the flux $J_i$. We interpret this model as a non-isothermal generalization of the 
Busenberg--Travis model. In the isothermal case $\theta_i=1$, we recover the regularized Busenberg--Travis equations
\begin{align}\label{1.BT}
  \pa_t\rho_i = \diver_x\bigg(\sigma\frac{\na_x\rho_i}{\nu_im_i}
  + \frac{\rho_i}{\nu_i}\sum_{j=1}^Na_{ij}\na_x\rho_j\bigg).
\end{align}
The classical Busenberg--Travis model is obtained after setting $\sigma=0$ and $a_{ij}=1$ \cite{BuTr83}. It describes the segregation of population species and can be derived from a mean-field-type limit from interacting particle systems \cite{CDJ19}. 

We can derive equations \eqref{1.BT} without the $\sigma$-diffusion term by changing the scaling: Setting $\sigma=\sqrt{\eps}$ and performing the modified Chapman--Enskog expansion $f_i^\eps=f_i^0 + \sqrt{\eps}g_i^\eps$ leads directly to \eqref{1.BT} with $\sigma=0$; see Section \ref{sec.hf}. We show in Appendix \ref{sec.scale} that this corresponds to a high-field scaling.

\subsection{State of the art}

Let us mention other cross-diffusion derivations from kinetic models in the literature. A cross-diffusion--fluid system was obtained from a kinetic--fluid model with turning and macroscopic interaction operators \cite{BKZ18}. Without fluid coupling, the cross-diffusion system resembles the Shigesada--Kawasaki--Teramoto system analyzed in \cite{ChJu04}. In a similar spirit as in \cite{BKZ18}, chemotaxis Keller--Segel-fluid equations were derived in \cite{BBTW15}. A kinetic model with nonlocal Fokker--Planck collision operator was studied in \cite{CFI25}; the limit of dominant collisions (without rescaling of time) leads to nonlocal Busenberg--Travis equations. A related quasilinear parabolic system for gas mixtures was derived in \cite{MaTr24} from kinetic equations with dominant elastic collisions. 

BGK operators are suggested as simple approximations of the highly complex Boltzmann collision operator. They should satisfy the same physical properties of the Boltzmann collision operator, like conservation of mass, momentum, and energy; the $H$-theorem; the indifferentiability principle (when all particle properties are identical, the multispecies description reduces to the single-species description); positivity of the temperature; and the Onsager reciprocal relations. The development of multispecies BGK models is delicate, and various consistent choices are possible. Up to our knowledge, none of the existing BGK-type models comply with all of the mentioned requirements. 

The first consistent BGK model was suggested in \cite{AAP02}. It satisfies the conservation laws, the $H$-theorem, and the uniqueness of the equilibrium solution, but the strong conditions $v_{ij}=v_i$ and $\theta_{ij}=\theta_i$ may be subject to criticism. A general model for two species was proposed in \cite{KPP17}, fulfilling the conservation laws and the $H$-theorem, later improved in \cite{BBGSP18} and extended to reactive mixtures in \cite{MST25}. Haack et al.\ \cite{HHM17} suggest the multispecies BGK model \eqref{1.mbgk}, satisfying the same properties as before. Here, the symmetry conditions $v_{ij}=v_{ji}$ and $\theta_{ij}=\theta_{ji}$ are supposed. Benilov \cite{Ben24} discusses the validity of the Onsager reciprocal relations. The Onsager relations hold in the model of \cite{GSB89} but not in the multispecies version of \cite{HHM17}. For more details, we refer to \cite{KPP17,Pup19}. 

\medskip
The paper is organized as follows. We define the multispecies BGK model of Gross and Krook in Section \ref{sec.mbgk} and derive formally the non-isothermal Maxwell--Stefan equations. The diffusion limit in the multispecies BGK model with Brinkman force is computed in Section \ref{sec.bgkb}. For both approaches, we compute the entropy equalities. In the appendices, we recall the computation of the multispecies velocities $v_{ij}$ and temperatures $\theta_{ij}$ and detail the physical background of our scaling. 


\section{Diffusion limit in a multispecies BGK model}\label{sec.mbgk}

The aim of this section is the formal derivation of non-isothermal Maxwell--Stefan equations from the multispecies BGK model of Gross and Krook \cite{GrKr56} in the diffusion limit, using a Chapman--Enskog expansion. The BGK model replaces the Boltzmann collision operator with a simpler nonlinear operator, still conserving mass, momentum, and energy and satisfying an $H$-theorem \cite[Sec.~IV.1]{Cer69}. Before presenting the Chapman--Enskog procedure, we introduce the multispecies BGK collision operator and the nonlinear BGK model. 

\subsection{Multispecies BGK collision operator}

Consider a gas mixture consisting of $N$ different particle species. Each particle of the $i$th species has the mass $m_i$, $i=1,\ldots,N$. The collision of two particles is assumed to be elastic and to conserve the mass of each particle. Let $\xi$ and $\xi_*$ be the pre-collisional velocities of species $i$ and $j$ and $\xi'$ and $\xi'_*$ be the corresponding post-collisional velocities. The collisions conserve the momentum and kinetic energy,
\begin{align}\label{bgk.me}
  m_i\xi' + m_j\xi'_* = m_i\xi + m_j\xi_*, \quad
  \frac{m_i}{2}|\xi'|^2 + \frac{m_j}{2}|\xi'_*|^2
  = \frac{m_i}{2}|\xi|^2 + \frac{m_j}{2}|\xi_*|^2.
\end{align}
These relations yield for $i,j=1,\ldots,N$,
\begin{align*}
  0 &= (m_i+m_j)\big(m_i|\xi|^2 + m_j|\xi_*|^2 - m_i|\xi'|^2 
  - m_j|\xi'_*|^2\big) \\
  &\phantom{xx}- \big(|m_i\xi+m_j\xi_*|^2
  - |m_i\xi'+m_j\xi'_*|^2\big) \\
  &= m_im_j\big(|\xi-\xi_*|^2 - |\xi'-\xi'_*|^2\big)
\end{align*}
and consequently $|\xi-\xi_*|=|\xi'-\xi'_*|$. Thus, there exists $\sigma\in\mathbb{S}^2$, where $\mathbb{S}^2$ is the two-dimensional unit sphere, such that $\xi'-\xi'_*=|\xi-\xi_*|\sigma$, and the conservation of momentum implies the relations
\begin{align*}
  \xi' &= \frac{m_i\xi'+m_j\xi'_*}{m_i+m_j}
  + \frac{m_j}{m_i+m_j}(\xi'-\xi'_*)
  = \frac{1}{m_i+m_j}\big(m_i\xi+m_j\xi_*+m_j|\xi-\xi_*|\sigma\big), \\
  \xi'_* &= \frac{m_i\xi'+m_j\xi'_*}{m_i+m_j}
  - \frac{m_i}{m_i+m_j}(\xi'-\xi'_*)
  = \frac{1}{m_i+m_j}\big(m_i\xi+m_j\xi_*-m_i|\xi-\xi_*|\sigma\big),
\end{align*}
which express $(\xi',\xi'_*)$ in terms of $(\xi,\xi_*)$ (and $\sigma$). The multispecies Boltzmann collision operator is the difference of gain and loss terms,
\begin{align*}
  & Q_i^B(f) = \sum_{i=1}^N Q_{ij}^B(f_i,f_j), \quad\mbox{where} \\
  & Q_{ij}^B(f_i,f_j) = \int_{\R^3}\int_{\mathbb{S}^2}B_{ij}(\xi,\xi_*,\sigma)
  (f'_if'_{j*}-f_if_{j*})d\sigma d\xi_*,
\end{align*}
describing the collisions of particles of species $i$ with particles of species $j$. Here, we abbreviated $f'_i=f_i(\xi')$, $f_{j*}=f_j(\xi_*)$, and $f'_{j*}=f_j(\xi'_*)$. The collision kernels (or cross-sections) $B_{ij}$ satisfy the micro-reversibility conditions $B_{ij}(\xi,\xi_*,\sigma)=B_{ji}(\xi_*,\xi,\sigma)$ and $B_{ij}(\xi,\xi_*,\sigma)=B_{ij}(\xi',\xi'_*,\sigma)$ for all $i\neq j$, $\xi$, $\xi_*\in\R^3$, and $\sigma\in\mathbb{S}^2$. The collision operator fulfills the following invariance properties that reflect the conservation of mass, momentum, and energy \cite[Sec.~3.3]{BGS15}:
\begin{align}
  \int_{\R^3}Q^B_{ij}(f_i^\eps,f_j^\eps)d\xi &= 0, \nonumber \\
  \int_{\R^3}Q^B_{ij}(f_i^\eps,f_j^\eps)m_i\xi d\xi 
  + \int_{\R^3}Q^B_{ji}(f_j^\eps,f_i^\eps)m_j\xi d\xi &= 0, 
  \label{bgk.Qmme} \\
  \int_{\R^3}Q^B_{ij}(f_i^\eps,f_j^\eps)m_i|\xi|^2 d\xi 
  + \int_{\R^3}Q^B_{ji}(f_j^\eps,f_i^\eps)m_j|\xi|^2 d\xi &= 0 \nonumber
\end{align}
for all $i\neq j$. The collision operator also satisfies a multispecies form of the $H$-theorem, i.e., the entropy 
\begin{align*}
  H(f) = -\sum_{i=1}^N\int_{\R^3}\int_{\R^3}f_i (\log f_i-1)d\xi dx
\end{align*}
fulfills $dH/dt\ge 0$ with equality if and only if
\begin{align*}
  f_i = n_i\bigg(\frac{m_i}{2\pi\theta_{\rm eq}}\bigg)^{3/2}
  \exp\bigg(-\frac{m_i}{2\theta_{\rm eq}}|\xi-v_{\rm eq}|^2\bigg)
\end{align*}
holds for some equilibrium number density $n_i$, equilibrium velocity $v_{\rm eq}$, and equilibrium temperature $\theta_{\rm eq}$ that are common to all species. Here, temperatures are measured in energy units so that the Boltzmann constant $k_B$ equals one. 

We wish to simplify the collision operator. To this end, we introduce the multispecies equilibrium (Maxwellian)
\begin{align*}
  M_{ij}(\xi) = n_i\bigg(\frac{m_i}{2\pi\theta_{ij}}\bigg)^{3/2}
  \exp\bigg(-\frac{m_i}{2\theta_{ij}}|\xi-v_{ij}|^2\bigg),
\end{align*}
where $n_i$ is the number density of the $i$th species, and the multispecies velocities $v_{ij}$ and multispecies temperatures $\theta_{ij}$ are determined by momentum and energy conservation \eqref{bgk.me}; see below. We require that $v_{ij}=v_{ji}$ and $\theta_{ij}=\theta_{ji}$. These assumptions are sufficient to ensure the $H$-theorem \cite[Sec.~3.2]{HHM17} and they imply, together with energy conservation \eqref{bgk.me}, that $M_{ij}'M_{ji*}'=M_{ij}M_{ji*}$. 

In the BGK approximation, we suppose that the particles reach the equilibrium after a single collision, i.e.\ $f_i'\approx M_{ij}'$ and $f'_*\approx M'_{ji*}$. Moreover, replacing $f_{j*}$ by $M_{ji*}$ (see below) leads to 
\begin{align*}
  Q^B_{ij}(f_i,f_j)\approx \int_{\R^3}\int_{\mathbb{S}^2} 
  B_{ij}(\xi,\xi_*,\sigma)(M_{ij}'M'_{ji*} - f_iM_{ji*})d\xi_* d\sigma.
\end{align*}
We deduce from $M_{ij}'M_{ji*}'=M_{ij}M_{ji*}$ that
\begin{align}\label{bgk.Q}
  Q^B_{ij}(f_i,f_j)\approx \nu_{ij}(M_{ij}-f_i), 
  \quad\mbox{where}\quad
  \nu_{ij}(\xi) = \int_{\R^3}\int_{\mathbb{S}^2}
  B_{ij}(\xi,\xi_*,\sigma)M_{ji*}d\xi_* d\sigma.
\end{align} 
The replacement of $f_{j*}$ by $M_{ji*}$ can be motivated if $B_{ij}$ only depends on $\sigma$ \cite[Sec.~1.1]{Pup19}, since in this case
\begin{align*}
  Q^B_{ij}(f_i,f_j)\approx \int_{\R^3}\int_{\mathbb{S}^2} 
  B_{ij}(\sigma)(M_{ij}M_{ji*} - f_if_{j*})d\xi_* d\sigma
  = \bigg(\int_{\mathbb{S}^2}B_{ij}d\sigma\bigg)n_j(M_{ij}-f_i),
\end{align*}
as $M_{ji*}$ and $f_{j*}$ have the same zeroth-order moment $n_j$. The approximation \eqref{bgk.Q} is the multispecies BGK collision operator. This formulation keeps the quadratic structure of the Boltzmann collision operator. To simplify the computations, we assume that $\nu_{ij}$ is independent of the velocity $\xi$. In the following, we consider only the multispecies BGK collision operator.

\subsection{Multispecies BGK model}\label{sec.bgkmodel}

We assume that the Knudsen number, representing the average frequency of collisions, and the Mach number, representing the ratio of flow velocity to the local sound speed, are small and of order $\eps>0$; see Appendix \ref{sec.scale}.
Then the multispecies BGK model in the diffusion scaling reads as
\begin{align}\label{bgk.eq}
  \eps\pa_t f_i^\eps + \xi\cdot\na_x f_i^\eps
  = \frac{1}{\eps}\sum_{j=1}^N Q_{ij}(f_i^\eps,f_j^\eps)
  \quad\mbox{in }\R^3,\ t>0,\ i=1,\ldots,N,
\end{align}
where the multispecies BGK operator is given by
\begin{align}\label{bgk.Qij}
  Q_{ij}(f_i^\eps,f_j^\eps) = \nu_{ij}(M_{ij}^\eps-f_i^\eps), \quad
  M_{ij}^\eps(\xi) = n_i^\eps\bigg(\frac{m_i}{2\pi\theta_{ij}^\eps}
  \bigg)^{3/2}\exp\bigg(-\frac{m_i}{2\theta_{ij}^\eps}
  |\xi-\eps v_{ij}^\eps|^2\bigg).
\end{align}
The multispecies velocities are of order $\eps$, since we are interested in pure diffusive dynamics. 

The moments of $f_i^\eps$ are the number density, momentum (supposed to be of order $\eps$), and energy of the $i$th species,
\begin{align*}
  \int_{\R^3}f_i^\eps d\xi = n_i^\eps, \quad
  \int_{\R^3}f_i^\eps \xi d\xi = \eps n_i^\eps v_i^\eps, \quad
  \int_{\R^3}f_i^\eps|\xi|^2 d\xi
  = \frac{3}{m_i}n_i^\eps\theta_i^\eps + \eps^2 n_i^\eps|v_i^\eps|^2.
\end{align*}
These relations define $n_i^\eps$, $v_i^\eps$, and $\theta_i^\eps$. A computation shows that
\begin{align}\label{bgk.MMeps}
  \int_{\R^3}M_{ij}^\eps d\xi = n_i^\eps, \quad
  \int_{\R^3}M_{ij}^\eps\xi d\xi = \eps n_i^\eps v_{ij}^\eps, \quad
  \int_{\R^3}M_{ij}^\eps|\xi|^2 d\xi 
  = \frac{3}{m_i}n_i^\eps\theta_{ij}^\eps 
  + \eps^2n_i^\eps|v_{ij}^\eps|^2.
\end{align}

We still need to determine the velocities $v_{ij}^\eps$ and temperatures $\theta_{ij}^\eps$. Imposing the invariance properties \eqref{bgk.Qmme} for $Q_{ij}$ instead of $Q_{ij}^B$ and recalling that we supposed the symmetry $v_{ij}^\eps=v_{ji}^\eps$ and $\theta_{ij}^\eps=\theta_{ji}^\eps$, a computation shows that
(see Lemma \ref{lem.thetaij} in Appendix \ref{sec.theta})
\begin{equation}\label{bgk.vij}
\begin{aligned}
  v_{ij}^\eps &= \alpha_{ij}^\eps v_i^\eps 
  + \alpha_{ji}^\eps v_j^\eps, \\
  \theta_{ij}^\eps &= (\beta_{ij}^\eps\theta_i^\eps
  + \beta_{ji}^\eps\theta_j^\eps)
  + \frac{\eps^2}{3}
  \big(\beta_{ij}^\eps m_i(|v_i^\eps|^2-|v_{ij}^\eps|^2)
  + \beta_{ji}^\eps m_j(|v_j^\eps|^2-|v_{ij}^\eps|^2)\big),
\end{aligned}
\end{equation}
where
\begin{align}\label{bgk.alpha}
  \alpha_{ij}^\eps = \frac{\nu_{ij}\rho_i^\eps}{
  \nu_{ij}\rho_i^\eps+\nu_{ji}\rho_j^\eps}, \quad
  \beta_{ij}^\eps = \frac{\nu_{ij}n_i^\eps}{
  \nu_{ij}n_i^\eps+\nu_{ji}n_j^\eps},
\end{align}
and we have introduced the mass densities $\rho_i^\eps=m_in_i^\eps$. Lemma \ref{lem.thetaij} also shows that the multispecies temperature is positive if $\theta_i^\eps$, $\theta_j^\eps$, and $\nu_{ij}n_i^\eps+\nu_{ji}n_j^\eps$ are positive. This was first proved in \cite[Prop.~2]{HHM17}.

According to \cite[Prop.~1]{HHM17}, the multispecies BGK model satisfies the $H$-theorem, i.e.\ $(dH/dt)(f)\ge 0$ and $(dH/dt)(f)=0$ if and only if $f_i = M_{ii}^\eps$ with $v_{ij}^\eps = v^\eps$ and $\theta_{ij}^\eps = \theta^\eps$, which means that 
\begin{align*}
  f_i = n_i^\eps\bigg(\frac{m_i}{2\pi\theta^\eps}\bigg)^{3/2}
  \exp\bigg(-\frac{m_i}{2\theta^\eps}|\xi-\eps v^\eps|^2\bigg).
\end{align*}
In the limit $\eps\to 0$, the equilibrium $M_{ii}^\eps$ converges to
\begin{align}\label{bgk.M0}
  M_i^0(\xi) = n_i\bigg(\frac{m_i}{2\pi\theta}\bigg)^{3/2}
  \exp\bigg(-\frac{m_i}{2\theta}|\xi|^2\bigg),
\end{align}
where $n_i=\lim_{\eps\to 0}n_i^\eps$ and $\theta=\lim_{\eps\to 0}\theta^\eps$. The moments of $M_i^0$ equal
\begin{equation}\label{bgk.MM}
\begin{aligned}
  & \int_{\R^3}M_i^0d\xi = n_i, \quad 
  \int_{\R^3}M_i^0\xi d\xi = \int_{\R^3}M_i^0\xi|\xi|^2 d\xi = 0, \\
  & \int_{\R^3}M_i^0(\xi\otimes\xi) d\xi 
  = \frac{n_i\theta}{m_i}\mathbb{I}, \quad
  \int_{\R^3}M_i^0(\xi\otimes\xi)|\xi|^2 d\xi 
  = 5\frac{n_i\theta^2}{m_i^2}\mathbb{I}
\end{aligned}
\end{equation}
where $\mathbb{I}$ is the unit matrix in $\R^{3\times 3}$. 

\subsection{Chapman--Enskog expansion}

We expand the solution $f_i^\eps$ to \eqref{bgk.eq} (which is assumed to exist) around the Maxwellian \eqref{bgk.M0},
\begin{align*}
  f_i^\eps = M_i^0 + \eps g_i^\eps, \quad i=1,\ldots,N,
\end{align*}
which in fact defines the correction $g_i^\eps$. We compute the zeroth-, first-, and second-order moment equations. First, because of mass conservation,
\begin{align*}
  \pa_t\int_{\R^3} f_i^\eps d\xi + \frac{1}{\eps}\diver_x
  \int_{\R^3}M_i^0\xi d\xi + \diver_x\int_{\R^3}g_i^\eps \xi d\xi
  = \frac{1}{\eps^2}\sum_{j=1}^N\int_{\R^3} \nu_{ij}(M_{ij}^\eps - f_i^\eps)d\xi
  = 0.
\end{align*}
The second term on the left-hand side vanishes since $\xi\mapsto M_i^0(\xi)\xi$ is odd. We conclude from
\begin{align*}
  \eps n_i^\eps v_i^\eps = \int_{\R^3}f_i^\eps \xi d\xi 
  = \int_{\R^3}M_i^0\xi d\xi + \eps\int_{\R^3}g_i^\eps\xi d\xi
  = \eps\int_{\R^3}g_i^\eps\xi d\xi
\end{align*}
that $\int_{\R^3}g_i^\eps\xi d\xi=n_i^\eps v_i^\eps$. Assuming the convergences $n_i^\eps\to n_i$, $v_i^\eps\to v_i$, and $g_i^\eps\to g_i^0$ as $\eps\to 0$ in a sense that allows us to perform the limit in the integrals, we find the mass conservation equation
\begin{align*}
  \pa_t n_i + \diver_x(n_iv_i) = 0.
\end{align*}

The first-order moment equation reads as
\begin{align*}
  \eps\pa_t\int_{\R^3} f_i^\eps\xi d\xi 
  &+ \diver_x\int_{\R^3}M_i^0(\xi\otimes\xi) d\xi
  + \eps\diver_x\int_{\R^3}g_i^\eps(\xi\otimes\xi) d\xi \\
  &= \frac{1}{\eps}\sum_{j=1}^N\nu_{ij}\int_{\R^3}(M_{ij}^\eps-f_i^\eps)
  \xi d\xi.
\end{align*}
The first and third terms on the left-hand side vanish in the limit $\eps\to 0$, while the second term equals $\na_x(n_i\theta)/m_i$ (see \eqref{bgk.MM}) and the right-hand side becomes (see \eqref{bgk.MMeps})
\begin{align*}
  \frac{1}{\eps}\sum_{j=1}^N\nu_{ij}\int_{\R^3}(M_{ij}^\eps-f_i^\eps)
  \xi d\xi = \frac{1}{\eps}\sum_{j=1}^N\nu_{ij}
  (\eps n_i^\eps v_{ij}^\eps - \eps n_i^\eps v_i^\eps)
  = \sum_{j=1}^N \nu_{ij}\alpha_{ji}^\eps n_i^\eps(v_j^\eps-v_i^\eps),
\end{align*}
since $v_{ij}^\eps-v_i^\eps = (\alpha_{ij}^\eps-1)v_i^\eps + \alpha_{ji}^\eps v_j^\eps = \alpha_{ji}^\eps(v_j^\eps-v_i^\eps)$. We obtain in the limit $\eps\to 0$ the Maxwell--Stefan equations
\begin{align*}
  \na_x(n_i\theta)
  = -\sum_{j=1}^N \nu_{ij}\alpha_{ji} m_in_i(v_i-v_j), \quad
  \alpha_{ji} = \frac{\nu_{ji}m_j n_j}{\nu_{ij}m_in_i+\nu_{ji}m_jn_j},
\end{align*}
which can be written as
\begin{align*}
  \na_x(n_i\theta)
  = -\sum_{j=1}^N D_{ij}n_in_j(v_i-v_j), \quad D_{ij} = \frac{\nu_{ij}\nu_{ji}m_im_j}{\nu_{ij}m_in_i+\nu_{ji}m_jn_j}.
\end{align*}
The symmetry of $(D_{ij})$ implies that the sum of these equations over $i=1,\ldots,N$ equals
\begin{align*}
  \sum_{i=1}^N\na_x(n_i\theta) = 0, 
\end{align*}
showing that $\sum_{i=1}^N n_i\theta$ is constant in space.

To derive the energy balance law, we compute the second-order moment equations:
\begin{align*}
  \eps^2\pa_t\int_{\R^3}f_i^\eps |\xi|^2d\xi
  &+ \eps\diver_x\int_{\R^3}M_i^0 \xi|\xi|^2d\xi
  + \eps^2\diver_x\int_{\R^3}g_i^\eps\xi|\xi|^2 d\xi \\
  &= \sum_{j=1}^N\nu_{ij}\int_{\R^3}(M_{ij}^\eps-f_i^\eps)|\xi|^2 d\xi.
\end{align*}
The first term on the left-hand side equals $\int_{\R^3}f_i^\eps |\xi|^2d\xi = 3n_i^\eps\theta_i^\eps/m_i + O(\eps^2)$ and the second term vanishes:
\begin{align}\label{bgk.mom2}
  3\pa_t\bigg(\frac{n_i^\eps\theta_i^\eps}{m_i}\bigg)
  + \diver_x\int_{\R^3}g_i^\eps\xi|\xi|^2 d\xi + O(\eps^2)
  = \frac{1}{\eps^2}\sum_{j=1}^N
  \nu_{ij}\int_{\R^3}(M_{ij}^\eps-f_i^\eps)|\xi|^2 d\xi.
\end{align}
We conclude that
\begin{align*}
  \sum_{j=1}^N\nu_{ij}\int_{\R^3}(M_{ij}^\eps-f_i^\eps)|\xi|^2 d\xi
  = O(\eps^2).
\end{align*} 
By \eqref{bgk.MM}, the left-hand side becomes
\begin{align}\label{bgk.M2}
  \sum_{j=1}^N\nu_{ij}\int_{\R^3}(M_{ij}^\eps-f_i^\eps)|\xi|^2 d\xi
  = \sum_{j=1}^N\nu_{ij}\bigg(
  \frac{3n_i^\eps}{m_i}(\theta_{ij}^\eps-\theta_i^\eps)
  + \eps^2 n_i^\eps(|v_{ij}^\eps|^2 - |v_i^\eps|^2)\bigg).
\end{align}
We deduce from definition \eqref{bgk.vij} of $\theta_{ij}^\eps$ that
\begin{align*}
  \theta_{ij}^\eps - \theta_i^\eps 
  = \beta_{ji}^\eps(\theta_j^\eps-\theta_i^\eps)
  + \frac{\eps^2}{3}
  \big(\beta_{ij}^\eps m_i(|v_i^\eps|^2-|v_{ij}^\eps|^2)
  + \beta_{ji}^\eps m_j(|v_j^\eps|^2-|v_{ij}^\eps|^2)\big).
\end{align*}
Hence, \eqref{bgk.M2} shows that $\theta_{i}^\eps-\theta_j^\eps = O(\eps^2)$, and as expected, there is only one mixture temperature in the limit, $\theta=\lim_{\eps\to 0}\theta_i^\eps$. In particular, $\theta_{ij}^\eps-\theta=O(\eps^2)$.

This order is not sufficient to pass to the limit $\eps\to 0$ in \eqref{bgk.M2}, divided by $\eps^2$. However, it turns out that the sum over $i=1,\ldots,N$ vanishes. Indeed, we have
\begin{align*}
  \sum_{i=1}^N m_i\sum_{j=1}^N\frac{3n_i^\eps}{m_i}
  \nu_{ij}\beta_{ji}^\eps(\theta_j^\eps-\theta_i^\eps) = 0,
\end{align*}
since $n_i^\eps\nu_{ij}\beta_{ji}^\eps=\nu_{ij}\nu_{ji}n_i^\eps n_j^\eps/(\nu_{ij}n_i^\eps+\nu_{ji}n_j^\eps)$ is symmetric. (We have used the fact that $\sum_{i,j=1}^nA_{ij}B_{ij}=0$ if $(A_{ij})$ is symmetric and $(B_{ij})$ is skew-symmetric.) Therefore, using $\beta_{ij}^\eps+\beta_{ji}^\eps=1$, the sum of \eqref{bgk.M2} over $i=1,\ldots,N$ becomes
\begin{align*}
  &\sum_{i,j=1}^N m_i\nu_{ij}\int_{\R^3}
  (M_{ij}^\eps-f_i^\eps)|\xi|^2 d\xi
  = 3\sum_{i,j=1}^n\nu_{ij} n_i^\eps
  \beta_{ji}^\eps(\theta_j^\eps-\theta_i^\eps) \\
  &\phantom{xxx}+ \eps^2\sum_{i,j=1}^n\nu_{ij}n_i^\eps\big(
  \beta_{ij}^\eps m_i(|v_i^\eps|^2-|v_{ij}^\eps|^2)
  + \beta_{ji}^\eps m_j(|v_j^\eps|^2-|v_{ij}^\eps|^2)
  + m_i(|v_{ij}^\eps|^2-|v_i^\eps|^2)\big) \nonumber \\
  &\phantom{x}= \eps^2\sum_{i,j=1}^n\nu_{ij}n_i^\eps\big(
  \beta_{ij}^\eps m_i(|v_i^\eps|^2-|v_{ij}^\eps|^2)
  + \beta_{ji}^\eps m_j(|v_j^\eps|^2-|v_{ij}^\eps|^2) \nonumber \\
  &\phantom{xxx}+ (\beta_{ij}^\eps+\beta_{ji}^\eps)
  m_i(|v_{ij}^\eps|^2-|v_{i}^\eps|^2)\big) \nonumber \\
  &\phantom{x}= \eps^2\sum_{i,j=1}^n\nu_{ij} n_i^\eps\beta_{ji}^\eps
  \big((m_j|v_j^\eps|^2-m_i|v_{i}^\eps|^2)
  + (m_i-m_j)|v_{ij}^\eps|^2\big) \nonumber \\
  &\phantom{x}= \eps^2\sum_{i,j=1}^n\nu_{ij} n_i^\eps\beta_{ji}^\eps
  \big[(\sqrt{m_j}v_j^\eps - \sqrt{m_i}v_i^\eps)
  (\sqrt{m_j}v_j^\eps + \sqrt{m_i}v_i^\eps)
  + (m_i-m_j)|v_{ij}^\eps|^2\big]. \nonumber 
\end{align*}
Since the matrices $(\nu_{ij} n_i^\eps\beta_{ji}^\eps(\sqrt{m_j}v_j^\eps + \sqrt{m_i}v_i^\eps))_{ij}$, $(\nu_{ij} n_i^\eps\beta_{ji}^\eps|v_{ij}^\eps|^2)_{ij}$ are symmetric and the matrices $(\sqrt{m_j}v_j^\eps - \sqrt{m_i}v_i^\eps)_{ij}$, $(m_i-m_j)_{ij}$ are skew-symmetric, the right-hand side vanishes. 

Hence, the sum of \eqref{bgk.mom2} over $i=1,\ldots,N$ can be written as
\begin{align*}
  3\sum_{i=1}^N\pa_t(n_i^\eps\theta_i^\eps)
  + \diver_x\sum_{i=1}^N m_i\int_{\R^3}g_i^\eps\xi|\xi|^2 d\xi = 0,
\end{align*}
and the limit $\eps\to 0$ gives
\begin{align}\label{bgk.mom22}
  3\sum_{i=1}^N\pa_t(n_i\theta) 
  + \diver_x\sum_{i=1}^N m_i\int_{\R^3}g_i^0\xi|\xi|^2 d\xi=0 
\end{align}

It remains to determine the third-order moment of $g_i^0$. For this, we consider the third-order moment equations:
\begin{align}\label{bgk.mom3}
  \eps\pa_t&\int_{\R^3}f_i^\eps\xi|\xi|^2 d\xi
  + \diver_x\int_{\R^3}M_i^0(\xi\otimes\xi)|\xi|^2 d\xi
  + \eps\diver_x\int_{\R^3}g_i^\eps(\xi\otimes\xi)|\xi|^2 d\xi \\
  &= \frac{1}{\eps}\sum_{j=1}^N\nu_{ij}\int_{\R^3}
  (M_{ij}^\eps-M_i^0)\xi|\xi|^2 d\xi
  - \sum_{j=1}^N\nu_{ij}\int_{\R^3}g_i^\eps\xi|\xi|^2 d\xi. \nonumber
\end{align}
The first and third terms are of order $O(\eps)$, and the second term becomes
\begin{align*}
  \diver_x\int_{\R^3}M_i^0(\xi\otimes\xi)|\xi|^2 d\xi
  = \frac{5}{m_i^2}\na_x(n_i\theta^2).
\end{align*}
The first term on the right-hand side of \eqref{bgk.mom3} can be computed explicitly. In fact, the integral over $M_i^0\xi|\xi|^2$ vanishes and, using the change of variables $z=(m_i/\theta_{ij}^\eps)^{1/2}(\xi-\eps v_{ij}^\eps)$,
\begin{align*}
  \frac{1}{\eps}\sum_{j=1}^N&\nu_{ij}\int_{\R^3}
  M_{ij}^\eps\xi|\xi|^2 d\xi \\
  &= \frac{n_i^\eps}{\eps}\sum_{j=1}^N\nu_{ij}\int_{\R^3}
  \frac{e^{-|z|^2/2}}{(2\pi)^{3/2}}
  \bigg(\bigg(\frac{\theta_{ij}^\eps}{m_i}\bigg)^{1/2}z
  + \eps v_{ij}^\eps\bigg)
  \bigg|\bigg(\frac{\theta_{ij}^\eps}{m_i}\bigg)^{1/2}z
  + \eps v_{ij}^\eps\bigg|^2 dz \\
  &= n_i^\eps\sum_{j=1}^N\frac{\nu_{ij}\theta_{ij}^\eps}{m_i}\int_{\R^3}
  \frac{e^{-|z|^2/2}}{(2\pi)^{3/2}}
  (|z|^2 + 2(z\otimes z))v_{ij}^\eps dz + O(\eps) \\
  &= \frac{5}{m_i}n_i^\eps\sum_{j=1}^N\nu_{ij}\theta_{ij}^\eps 
  v_{ij}^\eps + O(\eps)
  = \frac{5}{m_i}n_i^\eps\sum_{j=1}^N\nu_{ij}\theta_{ij}^\eps 
  (\alpha_{ij}^\eps v_i^\eps + \alpha_{ji}^\eps v_j^\eps) + O(\eps^2) \\
  &\to \frac{5\nu_i}{m_i}n_i\theta\bar{v}_i
\end{align*}
as $\eps\to 0$, where 
\begin{align}\label{bgk.barv}
  \bar{v}_i = \frac{\sum_{j=1}^N\nu_{ij}
  (\alpha_{ij}v_i+\alpha_{ji}v_j)}{\sum_{j=1}^N \nu_{ij}}, \quad
  \nu_i=\sum_{j=1}^N\nu_{ij}.
\end{align}
We infer from \eqref{bgk.mom3} in the limit $\eps\to 0$ that
\begin{align*}
  \frac{5}{m_i^2}\na_x(n_i\theta^2)
  = \frac{5\nu_i}{m_i}n_i\theta\bar{v}_i
  - \nu_i\int_{\R^3}g_i^0\xi|\xi|^2 d\xi,
\end{align*}
 and consequently,
\begin{align*}
  m_i\int_{\R^3}g_i^0\xi|\xi|^2 d\xi
  = 5n_i\theta\bar{v}_i
  - \frac{5}{m_i\nu_i}\na_x(n_i\theta^2).
\end{align*}
We conclude from \eqref{bgk.mom22} that
\begin{align*}
  3&\sum_{i=1}^N\pa_t(n_i\theta) - 5\sum_{i=1}^N\diver_x
  \bigg(\frac{\na_x(n_i\theta^2)}{m_i\nu_i}\bigg)
  + 5\sum_{i=1}^N\diver_x (n_i\theta \bar{v}_i ) = 0.
\end{align*}
Summarizing the previous computations leads to the following result.

\begin{theorem}[Non-isothermal Maxwell--Stefan model]
\label{thm.bgk}
The Chapman--Enskog expansion $f_i^\eps=M_i^0+\eps g_i^\eps$ in the multispecies BGK model \eqref{bgk.eq} and the formal limit $\eps\to 0$ yield the non-isothermal Maxwell--Stefan equations
\begin{align*}
  & \pa_t n_i + \diver_x(n_iv_i) = 0, \quad
  \na_x(n_i\theta) = -\sum_{j=1}^N D_{ij}n_i n_j(v_i-v_j), \\
  & \frac32\sum_{i=1}^N \pa_t(n_i\theta) 
  + \frac{5}{2}\sum_{i=1}^N\diver_x
  \bigg\{n_i\theta\bar{v}_i 
  - \frac{\na_x \left( n_i\theta^2 \right)}{m_i\nu_i} \bigg\} = 0,
\end{align*}
where $\bar{v}_i$ and $\nu_i$ are defined in \eqref{bgk.barv}, the diffusion coefficients are given by 
\begin{align*}
  D_{ij} = \frac{\nu_{ij}\nu_{ji}m_im_j}{
  \nu_{ij}m_in_i+\nu_{ji}m_jn_j},
\end{align*}
and formally, for $i=1,\ldots,N$,
\begin{align*}
  n_i = \lim_{\eps\to 0}n_i^\eps 
  = \lim_{\eps\to 0}\int_{\R^3}f_i^\eps d\xi, \quad
  n_iv_i = \lim_{\eps\to 0}\frac{1}{\eps}
  \int_{\R^3}f_i^\eps\xi d\xi, \quad
  \theta = \lim_{\eps\to 0}\frac{m_i}{3n_i^\eps}\int_{\R^3}
  f_i^\eps|\xi|^2 d\xi.
\end{align*}
\end{theorem}


\subsection{Entropy structure}

The kinetic entropy 
\begin{align}\label{bgk.Heps}
  H^\eps = -\sum_{i=1}^N\int_{\R^3}\int_{\R^3}
  f_i^\eps(\log f_i^\eps-1)d\xi dx
\end{align}
is nondecreasing in time; this is known as the $H$-theorem, and it follows from
\begin{align*}
  \frac{dH^\eps}{dt} 
  &= -\sum_{i=1}^N\int_{\R^3}\int_{\R^3}
  \pa_t f_i^\eps\log f_i^\eps d\xi dx
  = -\frac{1}{\eps^2}\int_{\R^3}\int_{\R^3}\sum_{i,j=1}^N\nu_{ij}
  (M_{ij}^\eps-f_i^\eps)\log f_i^\eps d\xi dx \\
  &= \frac{1}{\eps^2}\int_{\R^3}\int_{\R^3}\sum_{i,j=1}^N\nu_{ij}
  (M_{ij}^\eps-f_i^\eps)(\log M_{ij}^\eps - \log f_i^\eps) d\xi dx
  \ge 0,
\end{align*}
where the last equality follows from the fact that $M_{ij}^\eps$ and $f_i^\eps$ have the same moments and the inequality is a consequence of the monotonicity of the logarithm. We claim that such a property holds under some conditions for the macroscopic entropy
\begin{align*}
  H^\eps\to H^0 = -\sum_{i=1}^N\int_{\R^3}\bigg\{n_i(\log n_i-1)
  + \frac32 n_i \bigg(\log\frac{m_i}{2\pi\theta}-1\bigg)\bigg\}dx.
\end{align*}

\begin{proposition}[Entropy inequality]
Let $m_i\nu_i=1$ for $i=1,\ldots,N$. Then
\begin{align*}
  \frac{dH^0}{dt} = \int_{\R^3}\bigg(\frac{1}{2\theta}
  \sum_{i,j=1}^N D_{ij}n_in_j|v_i-v_j|^2
  + \frac52\sum_{i=1}^n \frac{n_i}{\theta}|\na_x\theta|^2\bigg)dx \ge 0.
\end{align*}
\end{proposition}

\begin{proof}
We compute
\begin{align*}
  \frac{dH^0}{dt} &= -\sum_{i=1}^N\int_{\R^3} \bigg\{\pa_t n_i\bigg(\log n_i
  - \frac32\log\theta\bigg) - \frac32\frac{1}{\theta}\pa_t(n_i\theta)
  + \frac32\pa_t n_i\bigg\}dx \\
  &= -\sum_{i=1}^N\int_{\R^3}\bigg\{n_iv_i\cdot
  \bigg(\frac{\na_x n_i}{n_i}-\frac32\frac{\na_x\theta}{\theta}\bigg)
  + \frac52\bigg(n_i\theta\bar{v}_i 
  - \na_x\bigg(\frac{n_i\theta^2}{m_i\nu_i}\bigg)\bigg)
  \cdot\frac{\na_x\theta}{\theta^2}\bigg\}dx \\
  &= -\sum_{i=1}^N\int_{\R^3}\bigg\{v_i\cdot\bigg(\na_x n_i
  + \frac{n_i}{\theta}\na_x\theta\bigg)
  - \frac52\frac{n_i}{\theta}\na_x\theta\cdot(v_i-\bar{v}_i) \\
  &\phantom{xx}
  - \frac52\frac{\theta\na_x(n_i\theta)+n_i\theta\na_x\theta}{m_i\nu_i}
  \cdot\frac{\na_x\theta}{\theta^2}\bigg\}dx \\
  &= -\sum_{i=1}^N\int_{\R^3}\bigg(
  v_i\cdot\frac{\na_x(n_i\theta)}{\theta}
  - \frac52\frac{n_i}{\theta}\na_x\theta\cdot(v_i-\bar{v}_i)
  - \frac52\frac{n_i}{\theta}\frac{|\na_x\theta|^2}{m_i\nu_i}
  \bigg)dx,
\end{align*}
where we used the fact that $\sum_{i=1}^N\na_x(n_i\theta)/(m_i\nu_i)=0$, since $m_i\nu_i=1$. It follows from definition \eqref{bgk.barv} of $\bar{v}_i$ that
\begin{align*}
  v_i-\bar{v}_i 
  = \frac{1}{\nu_i}\sum_{j=1}^N\nu_{ij}\big((1-\alpha_{ij})v_i
  - \alpha_{ji}v_j\big)
  = \frac{1}{\nu_i}\sum_{j=1}^N\nu_{ij}\alpha_{ji}(v_i-v_j).
\end{align*}
Inserting the relation $\na_x(n_i\theta)=-\sum_{j=1}^N D_{ij}n_in_j(v_i-v_j)$, we deduce from the symmetry of $D_{ij}$ that 
\begin{align}
  \frac{dH^0}{dt} &= \int_{\R^3}\bigg(
  \frac{1}{\theta}\sum_{i,j=1}^N D_{ij}n_in_j(v_i-v_j)\cdot v_i 
  + \frac52\frac{\na_x\theta}{\theta}\cdot\sum_{i,j=1}^N
  \frac{\nu_{ij}\alpha_{ji}m_in_i}{m_i\nu_i}(v_i-v_j)dx \label{bt.aux} \\
  &\phantom{xx}+ \frac52\sum_{i=1}^N
  \frac{n_i}{\theta}\frac{|\na_x\theta|^2}{m_i\nu_i}\bigg)dx \nonumber \\
  &= \int_{\R^3}\frac{1}{2\theta}\sum_{i,j=1}^N D_{ij}n_in_j
  |v_i-v_j|^2dx + \frac52\int_{\R^3}\sum_{i,j=1}^N
  \nu_{ij}\alpha_{ji}m_in_i(v_i-v_j)\cdot 
  \frac{\na_x\theta}{\theta}dx \nonumber \\
  &\phantom{xx}+ \frac52\int_{\R^3}\sum_{i=1}^N
  \frac{n_i}{\theta}|\na_x\theta|^2dx, \nonumber 
\end{align}
where we used the assumption $m_i\nu_i=1$. The matrix with entries
\begin{align*}
  \nu_{ij}\alpha_{ji}m_in_i = \frac{\nu_{ij}\nu_{ji}(m_in_i)(m_jn_j)}{
  \nu_{ij}m_in_i + \nu_{ji}m_jn_j}
\end{align*}
is symmetric, while the matrix with entries $v_i-v_j$ is skew-symmetric. Hence, the second term on the right-hand side of \eqref{bt.aux} vanishes, and we conclude the proof.
\end{proof}

The assumption $m_i\nu_i=1$ is {\em not} needed in the isothermal case, which is not surprising since this product does not appear explicitly in the isothermal model. 

\begin{remark}[Rigorous limit]\rm
The diffusion limit $\eps\to 0$ was made rigorous in \cite{BrGr23} for the linearized Boltzmann equations, leading to the isothermal Maxwell--Stefan equations, and in \cite{CMPS04} for a kinetic velocity-jump model, leading to a drift-diffusion equation. It seems to be delicate to perform the rigorous limit $\eps\to 0$ in our model because of the highly nonlinear structure of the multispecies collision operator. 
\end{remark}

\begin{remark}[Boundary conditions]\rm
The boundary conditions for the macroscopic equations depend on the boundary conditions for the kinetic model. The expansion for higher-order quantities may not satisfy the kinetic boundary conditions, which requires the Knudsen-layer correction near the boundary \cite{TaJu08}. It was shown in \cite[Prop.~2]{Pla19} that, assuming a linear combination of a regular reflection with a diffusive boundary operator, the leading part of the expansion satisfies no-flux boundary conditions. 
\end{remark}


\section{Diffusion limit in a BGK model with Brinkman force term}
\label{sec.bgkb}

In this section, we derive the generalized Busenberg--Travis equations from a multispecies kinetic model in the diffusion limit. Compared to the previous section, we choose a single collision operator similarly as in \cite{AAP02} and include a Brinkman force term. 

\subsection{BGK--Brinkman model}

The equations in the diffusion scaling read as
\begin{align}\label{br.eq}
  \frac{\eps}{\sigma}\pa_t f_i^\eps + \xi\cdot\na_x f_i^\eps
  + \frac{1}{\sigma}\na_x\phi_i^\eps\cdot\na_\xi f_i^\eps
  = \frac{\nu_i}{\eps}(M_i(f_i^\eps)-f_i^\eps) \quad\mbox{in }\R^3,
  \ t>0,
\end{align}
for $i=1,\ldots,N$, where $\eps>0$ is associated to the Knudsen and Mach number (see Appendix \ref{sec.scale}), $\sigma>0$ is a free parameter (see Section \ref{sec.hf}), $\nu_i>0$ is the collision frequency of the $i$th species, and $\phi_i^\eps$ is the Brinkman potential solving
\begin{align}\label{br.pot}
  -\eps\Delta_x\phi_i^\eps + \phi_i^\eps 
  = -\sum_{i=1}^N a_{ij}\rho_j^\eps\theta_j^\eps
  \quad\mbox{in }\R^3,
\end{align} 
where $a_{ij}\in\R$, the number density $n_i^\eps$ and temperature $\theta_i^\eps$ are given by 
\begin{align*}
  \int_{\R^3}f_i^\eps d\xi = n_i^\eps, \quad 
  \int_{\R^3}f_i^\eps|\xi|^2 d\xi = \frac{3}{m_i}n_i^\eps\theta_i^\eps,
\end{align*}
and $\rho_i^\eps=m_i n_i^\eps$ is the particle density. The Brinkman law is a generalization of the Darcy law $V_i=-\na_x p_i$, where $V_i$ is the partial velocity and $p_i=\sum_{i=1}^N a_{ij}\rho_j^\eps\theta_j^\eps$ is the partial pressure, by adding an elliptic regularization, $-\eps\Delta_x V_i+V_i=-\na_x p_i$. This equation follows from \eqref{br.pot} after differentiating and setting $V_i=\na_x\phi_i^\eps$. The Maxwellian in \eqref{br.eq} is given by
\begin{align*}
  (M_i(f_i^\eps))(\xi) = n_i^\eps\bigg(\frac{m_i}{2\pi\theta_i^\eps}\bigg)^{3/2}
  \exp\bigg(-\frac{m_i|\xi|^2}{2\theta_i^\eps}\bigg).
\end{align*}
This means that $M_i(f_i^\eps)$ has the same moments as $f_i^\eps$, 
\begin{align*}
  \int_{\R^3}M_i(f_i^\eps) d\xi = n_i^\eps, \quad 
  \int_{\R^3}M_i(f_i^\eps) \xi d\xi = 0, \quad
  \int_{\R^3}M_i(f_i^\eps)|\xi|^2 d\xi 
  = \frac{3}{m_i}n_i^\eps\theta_i^\eps,
\end{align*}
and $M_i(f_i^\eps)$ depends on $f_i^\eps$ in a nonlinear and nonlocal way. Compared to the model in \cite{AAP02}, we have set the velocity in the Maxwellian equal to zero and we have used the same temperatures for $M_i(f_i^\eps)$ and $f_i^\eps$.

\subsection{Chapman--Enskog expansion}

We expand the solution $f_i^\eps$ to \eqref{br.eq} as follows: 
\begin{align*}
  f_i^\eps = f_i^0 + \eps g_i^\eps, \quad i=1,\ldots,N,
\end{align*}
which defines the function $g_i^\eps$. The limit $\eps\to 0$ in \eqref{br.eq} leads to $M_i(f_i^0)=f_i^0$, where
\begin{align*}
  (M_i(f_i^0))(\xi) =
  n_i\bigg(\frac{m_i}{2\pi\theta_i}\bigg)^{3/2}
  \exp\bigg(-\frac{m_i|\xi|^2}{2\theta_i}\bigg),
\end{align*}
assuming that $n_i^\eps\to n_i$, $\theta_i^\eps\to \theta_i$, and $\phi_i^\eps\to\phi_i$, where $\phi_i=-\sum_{j=1}^N a_{ij}\rho_j\theta_j$. We insert the expansion $f_i^\eps= M_i(f_i^0) + \eps g_i^\eps$ into equation \eqref{br.eq},
\begin{align*}
  \eps\pa_t f_i^\eps + \sigma\xi\cdot\na_x(M_i(f_i^0)+\eps g_i^\eps)
  + \na_x\phi_i^\eps\cdot\na_\xi(M_i(f_i^0)+\eps g_i^\eps)
  = -\sigma\nu_i g_i^\eps,
\end{align*}
and perform the formal limit $\eps\to 0$,
\begin{align}\label{br.gi0}
  \sigma\xi\cdot\na_x M_i(f_i^0) + \na_x\phi_i\cdot\na_\xi M_i(f_i^0)
  = -\sigma\nu_i g_i^0,
\end{align}
where $g_i^0=\lim_{\eps\to 0}g_i^\eps$. 

The zeroth-order moment equation of \eqref{br.eq} becomes
\begin{align*}
  \pa_t\int_{\R^3}f_i^\eps d\xi
  &+ \frac{\sigma}{\eps}\diver_x\int_{\R^3}M_i(f_i^\eps)\xi d\xi
  + \sigma\diver_x\int_{\R^3}g_i^\eps\xi d\xi
  + \frac{1}{\eps}\na_x\phi_i^\eps\cdot\int_{\R^3}
  \na_\xi f_i^\eps d\xi \\
  &= -\frac{\sigma\nu_i}{\eps^2}\int_{\R^3}
  (M_i(f_i^\eps)-f_i^\eps)d\xi = 0.
\end{align*}
The second and fourth terms on the left-hand side vanish, and the limit $\eps\to 0$ yields
\begin{align*}
  \pa_t n_i + \sigma\diver_x\int_{\R^3}g_i^0\xi d\xi = 0.
\end{align*}
By inserting \eqref{br.gi0}, we compute the first-order moment of $g_i^0$:
\begin{align}\label{br.g1}
  \sigma\int_{\R^3}g_i^0\xi d\xi &= -\frac{1}{\nu_i}
  \bigg(\sigma\diver_x\int_{\R^3}
  M_i(f_i^0)(\xi\otimes\xi)d\xi + \na_x\phi_i\cdot\int_{\R^3}
  \na_\xi M_i(f_i^0)\xi d\xi\bigg) \\
  &= -\frac{\sigma}{\nu_i}\na_x\bigg(\frac{n_i\theta_i}{m_i}\bigg)
  + \frac{1}{\nu_i}n_i\na_x\phi_i. \nonumber 
\end{align}
This yields the equation
\begin{align*}
  \pa_t(m_in_i) = \diver_x\big(\nu_i^{-1}\big(\sigma\na_x(n_i\theta_i)
  - m_in_i\na_x\phi_i\big)\big).
\end{align*}

To derive the energy balance law, we consider the second-order moment equation of \eqref{br.eq}:
\begin{align*}
  \pa_t&\int_{\R^3}f_i^\eps|\xi|^2 d\xi
  + \frac{\sigma}{\eps}\diver_x\int_{\R^3}M_i(f_i^\eps)\xi|\xi|^2 d\xi
  + \sigma\diver_x\int_{\R^3}g_i^\eps\xi|\xi|^2 d\xi \\
  &\phantom{xx}+ \frac{1}{\eps}\na_x\phi_i^\eps\cdot
  \int_{\R^3}\na_\xi M_i(f_i^\eps)|\xi|^2 d\xi
  + \na_x\phi_i^\eps\cdot\int_{\R^3}\na_\xi g_i^\eps|\xi|^2 d\xi \\
  &= \frac{\sigma\nu_i}{\eps^2}\int_{\R^3}
  (M_i(f_i^\eps)-f_i^\eps)|\xi|^2d\xi = 0,
\end{align*}
since $M_i(f_i^\eps)$ and $f_i^\eps$ have the same moments. Taking into account that the second and fourth integrals on the left-hand side vanish, the limit $\eps\to 0$ gives
\begin{align*}
  \pa_t\bigg(\frac{3n_i\theta_i}{m_i}\bigg)
  + \sigma\diver_x\int_{\R^3}g_i^0\xi|\xi|^2 d\xi
  - 2\na_x\phi_i\cdot\int_{\R^3}g_i^0\xi d\xi = 0.
\end{align*}
The first-order moment of $g_i^0$ was computed in \eqref{br.g1}. It remains to compute the third-order moment of $g_i^0$:
\begin{align*}
  \sigma&\int_{\R^3}g_i^0\xi|\xi|^2 d\xi
  = -\frac{1}{\nu_i}\bigg(\sigma\diver_x\int_{\R^3}M_i(f_i^0)
  (\xi\otimes\xi)|\xi|^2 d\xi + \na_x\phi_i\cdot\int_{\R^3}
  \na_\xi M_i(f_i^0)\xi|\xi|^2 d\xi\bigg) \\
  &= -\frac{1}{\nu_i}\bigg(\sigma\diver_x\int_{\R^3}M_i(f_i^0)
  (\xi\otimes\xi)|\xi|^2 d\xi - \int_{\R^3}
  M_i(f_i^0)(|\xi|^2\mathbb{I}+2\xi\otimes\xi)\na_x\phi_i d\xi\bigg) \\
  &= -\frac{5}{\nu_i}\bigg(\sigma\na_x\bigg(\frac{n_i\theta_i^2}{m_i^2}
  \bigg) - \frac{n_i\theta_i}{m_i}\na_x\phi_i\bigg).
\end{align*}
We conclude that 
\begin{align*}
  \pa_t\bigg(\frac{3n_i\theta_i}{m_i}\bigg)
  &= 5\diver_x\bigg\{\frac{1}{\nu_i}\bigg(\sigma\na_x\bigg(
  \frac{n_i\theta_i^2}{m_i^2}\bigg) 
  - \frac{n_i\theta_i}{m_i}\na_x\phi_i\bigg)\bigg\} \\
  &\phantom{xx}
  + \frac{2}{\nu_i}\bigg(-\na_x\bigg(\frac{n_i\theta_i}{m_i}\bigg)
  +\frac{n_i}{\sigma}\na_x\phi_i\bigg)\cdot\na_x\phi_i,
\end{align*}
and, using $\rho_i=m_in_i$, we obtain the following result.

\begin{theorem}[Non-isothermal Busenberg--Travis model]
\label{thm.BT}
The Chapman--Enskog expansion $f_i^\eps=M_i^0+\eps g_i^\eps$ in the BGK--Brinkman model \eqref{br.eq} and the formal limit $\eps\to 0$ yield the non-isothermal Busenberg--Travis-type equations
\begin{align*}
  \pa_t\rho_i &= \diver_x\bigg\{\frac{\sigma}{\nu_i}
  \na_x\bigg(\frac{\rho_i\theta_i}{m_i}\bigg) - \frac{\rho_i}{\nu_i}\na_x\phi_i\bigg\},
  \quad \phi_i = -\sum_{i=1}^N a_{ij}\rho_j\theta_j, \\
  \frac32\pa_t\bigg(\frac{\rho_i\theta_i}{m_i}\bigg) 
  &= \frac52\diver_x\bigg\{\frac{\sigma}{\nu_i}\na_x
  \bigg(\frac{\rho_i\theta_i^2}{m_i^2}\bigg)
  - \frac{\rho_i\theta_i}{\nu_i m_i}\na_x\phi_i\bigg\} \\
  &\phantom{xx}- \frac{1}{\sigma}\bigg\{\frac{\sigma}{\nu_i}
  \na_x\bigg(\frac{\rho_i\theta_i}{m_i}\bigg) 
  - \frac{\rho_i}{\nu_i}\na_x\phi_i\bigg\}\cdot\na_x\phi_i,
\end{align*}
where
\begin{align*}
  \rho_i = \lim_{\eps\to 0}m_i n_i^\eps 
  = \lim_{\eps\to 0}\int_{\R^3}f_i^\eps d\xi, \quad
 \theta_i = \lim_{\eps\to 0}\frac{m_i}{3n_i^\eps}\int_{\R^3}
  f_i^\eps|\xi|^2 d\xi \quad\mbox{for }i=1,\ldots,N.
\end{align*}
\end{theorem}


\subsection{High-field scaling}\label{sec.hf}

We claim that the diffusion term in the mass balance equation \eqref{1.BT} vanishes if we make a high-field scaling. We choose $\sigma=\sqrt{\eps}$ in \eqref{br.eq}: 
\begin{align}\label{br.sigma}
  \sqrt{\eps}\pa_t f_i^\eps + \xi\cdot\na_x f_i^\eps
  + \frac{1}{\sqrt{\eps}}\na_x\phi_i\cdot\na_\xi f_i^\eps
  = \frac{\nu_i}{\eps}(M_i(f_i^\eps)-f_i^\eps) \quad\mbox{in }\R^3,
  \ t>0,
\end{align}
and expand $f_i^\eps=f_i^0+\sqrt{\eps}g_i^\eps$. As in the previous section, we find in the limit $\eps\to 0$ that $f_i^0=M_i(f_i^0)$. Inserting the expansion into \eqref{br.sigma} gives 
\begin{align*}
  \sqrt{\eps}\pa_t f_i^\eps 
  + \xi\cdot\na_x(M_i(f_i^0)+\sqrt{\eps}g_i^\eps)
  + \frac{1}{\sqrt{\eps}}\na\phi_i\cdot\na_\xi
  (M_i(f_i^0)+\sqrt{\eps}g_i^\eps)
  = -\frac{\nu_i}{\sqrt{\eps}}g_i^\eps,
\end{align*}
and the limit $\eps\to 0$ yields
\begin{align}\label{br.g0}
  g_i^0 = -\frac{1}{\nu_i}\na_x\phi_i\cdot\na_\xi M_i(f_i^0).
\end{align}
The zeroth-order moment equation becomes
\begin{align*}
  \pa_t\int_{\R^3}f_i^\eps d\xi 
  &+ \frac{1}{\sqrt{\eps}}\diver_x\int_{\R^3}M_i(f_i^0)\xi d\xi
  + \diver_x\int_{\R^3}g_i^\eps \xi d\xi \\
  &+ \frac{1}{\eps}\na_x\phi_i\cdot\int_{\R^3}\na_\xi f_i^\eps d\xi
  = -\frac{\nu_i}{\eps^{3/2}}\int_{\R^3}g_i^\eps d\xi = 0.
\end{align*}
Taking into account that the second and fourth terms vanish, the limit $\eps\to 0$ leads to the mass conservation law
\begin{align*}
  \pa_t n_i + \diver_x\int_{\R^3}g_i^0 \xi d\xi = 0.
\end{align*}
The flux is computed by inserting expression \eqref{br.g0} for $g_i^0$:
\begin{align*}
  \int_{\R^3}g_i^0 \xi d\xi = -\frac{1}{\nu_i}\na_x\phi_i\cdot
  \int_{\R^3}\na_\xi M_i(f_i^0)\xi d\xi
  = \frac{1}{\nu_i}\na_x\phi_i\cdot\int_{\R^3}M_i(f_i^0)d\xi
  = \frac{1}{\nu_i}n_i\na_x\phi_i.
\end{align*}
We deduce from $\phi_i=-\sum_{i=1}^N a_{ij}\rho_j\theta_j$ the Busenberg--Travis equations
\begin{align*}
  \pa_t\rho_i = -\diver_x\bigg(\frac{\rho_i}{\nu_i}\na_x\phi_i\bigg)
  = \diver_x\bigg(\frac{\rho_i}{\nu_i}
  \sum_{i=1}^N a_{ij}\na_x(\rho_j\theta_j)\bigg), \quad i=1,\ldots,N.
\end{align*}
For constant temperature $\theta_j=1$, we recover the generalized Busenberg--Travis model \cite{BuTr83}. (More precisely, the model in \cite{BuTr83} has been suggested for the special case $a_{ij}=1$.) Note that the high-field scaling does not allow us to derive the energy equation, since the second-order moment equation becomes
\begin{align*}
  \pa_t\int_{\R^3}f_i^\eps|\xi|^2d\xi 
  &+ \diver_x\int_{\R^3}g_i^\eps\xi|\xi|^2 d\xi
  - \frac{2}{\sqrt{\eps}}\na_x\phi_i\cdot\int_{\R^3}
  g_i^\eps\xi d\xi \\
  &= \frac{\nu_i}{\eps^{3/2}}\int_{\R^3}(M_i(f_i^0)-f_i^\eps)
  |\xi|^2 d\xi = 0,
\end{align*}
and the third term on the left-hand side does not converge as $\eps\to 0$.


\subsection{Entropy structure}

As in the first model, the kinetic entropy in definition \eqref{bgk.Heps} is nondecreasing in time:
\begin{align*}
  & \frac{dH^\eps}{dt} 
  = \frac{1}{\eps^2}\sum_{i=1}^N \int_{\R^3}\int_{\R^3}\nu_{i}
  (M_{i}^\eps-f_i^\eps)(\log M_{i}^\eps - \log f_i^\eps) d\xi dx
  \ge 0.
\end{align*}
We show that the macroscopic entropy
\begin{align*}
  H^0 = \lim_{\eps\to 0}H^\eps 
  = -\sum_{i=1}^N\int_{\R^3}\bigg(n_i(\log n_i-1)
  - \frac32n_i\log\frac{2\pi\theta_i}{m_i} - 1\bigg)dx
\end{align*}
is also nondecreasing in time under some positive semidefiniteness condition.

\begin{proposition}[Macroscopic entropy equality]
Let the symmetric part of $(a_{ij}/(m_i\nu_i\theta_i))_{ij}$ be positive semidefinite. Then the macroscopic entropy is nondecreasing in time:
\begin{align*}
  \frac{dH^0}{dt} &= \sum_{i=1}^N\int_{\R^3}\frac{1}{\nu_i}
  \bigg(\frac{\sigma}{m_i^2}
  \frac{|\na_x(\rho_i\theta_i)|^2}{\rho_i\theta_i}
  + \frac52\frac{\sigma}{m_i^2}\frac{\rho_i}{\theta_i}|\na_x\theta_i|^2
  + \frac{1}{\sigma}\frac{\rho_i}{\theta_i}|\na_x\phi_i|^2\bigg)dx \\
  &\phantom{xx}+ 2\sum_{i,j=1}^n\int_{\R^3}
  \frac{a_{ij}}{m_i\nu_i\theta_i}
  \na_x(\rho_i\theta_i)\cdot\na_x(\rho_j\theta_j)dx \ge 0.
\end{align*}
\end{proposition}

\begin{proof}
We use the evolution equations in Theorem \ref{thm.BT} to compute
\begin{align*}
  \frac{dH^0}{dt} &= \sum_{i=1}^N\int_{\R^3}\bigg(-\pa_t n_i\log n_i
  + \frac32\pa_t n_i\log\theta_i + \frac32\frac{1}{\theta_i}
  \pa_t(n_i\theta_i) - \frac32\pa_t n_i\bigg)dx \\
  &= \sum_{i=1}^N\int_{\R^3}\frac{1}{\nu_i}\bigg\{
  \bigg(\frac{\sigma}{m_i}\na_x(n_i\theta_i) - n_i\na_x\phi_i\bigg)
  \cdot\bigg(\frac{\na_x n_i}{n_i} 
  - \frac32\frac{\na_x\theta_i}{\theta_i}\bigg) \\
  &\phantom{xx}+ \frac52\bigg(\frac{\sigma}{m_i}\theta_i
  \na_x(n_i\theta_i) + \frac{\sigma}{m_i}n_i\theta_i\na_x\theta_i
  - n_i\theta_i\na_x\phi_i\bigg)\cdot\frac{\na_x\theta_i}{\theta_i^2} \\
  &\phantom{xx}- \frac{1}{\theta_i}
  \bigg(\na_x(n_i\theta_i) - \frac{m_in_i}{\sigma}\na_x\phi_i\bigg)
  \cdot\na_x\phi_i\bigg\}dx.
\end{align*}
We collect the terms with factors $\na_x(n_i\theta_i)$ and $\na_x\phi_i$:
\begin{align*}
  \frac{dH^0}{dt} &= \sum_{i=1}^N\int_{\R^3}\frac{1}{\nu_i}\bigg\{
  \frac{\sigma}{m_i}\na_x(n_i\theta_i)\cdot
  \bigg(\frac{\na_x n_i}{n_i} + \frac{\na_x\theta_i}{\theta_i}\bigg)
  + \frac52\frac{\sigma}{m_i}\frac{n_i}{\theta_i}|\na_x\theta_i|^2 \\
  &\phantom{xx}- 2\na_x\phi_i\cdot\bigg(\na_x n_i + \frac{n_i}{\theta_i}
  \na_x\theta_i\bigg) + \frac{m_i}{\sigma}\frac{n_i}{\theta_i}
  |\na_x\phi_i|^2\bigg\}dx \\
  &= \sum_{i=1}^N\int_{\R^3}\frac{1}{\nu_i}\bigg(
  \frac{\sigma}{m_i}\frac{|\na_x(n_i\theta_i)|^2}{n_i\theta_i}
  + \frac52\frac{\sigma}{m_i}\frac{n_i}{\theta_i}|\na_x\theta_i|^2 \\
  &\phantom{xx}
  - \frac{2}{\theta_i}\na_x\phi_i\cdot\na_x(n_i\theta_i)
  + \frac{m_i}{\sigma}\frac{n_i}{\theta_i}|\na_x\phi_i|^2\bigg)dx.
\end{align*}
Inserting the definition of $\phi_i$, the third term on the right-hand side is written as
\begin{align*}
  -\sum_{i=1}^N\int_{\R^3}\frac{2}{\nu_i}\frac{1}{\theta_i}
  \na_x\phi_i\cdot\na_x(n_i\theta_i)
  = 2\sum_{i,j=1}^n\int_{\R^3}\frac{a_{ij}}{m_i\nu_i\theta_i}
  \na_x(\rho_j\theta_j)\cdot\na_x(\rho_i\theta_i)dx,
\end{align*}
and this expression is nonnegative if the symmetric part of $(a_{ij}/(m_i\nu_i\theta_i))$ is positive definite.
\end{proof}

\begin{remark}[Rao entropy]\rm
It is shown in \cite{CCDJ24} that the so-called Rao entropy 
\begin{align*}
  E_R = \frac12\sum_{i,j=1}^N\int_{\R^3}a_{ij}\rho_i\rho_j dx
\end{align*}
can be interpreted as the potential energy associated to the compressible Navier--Stokes equations with force $-\rho_i\na\phi_i$ in the zero-inertia limit, and this functional is nonincreasing in time along solutions to the classical (isothermal) Busenberg--Travis model. Such an interpretation is also possible on the kinetic level under the assumptions that $(a_{ij})$ is symmetric and positive definite and $m_i\nu_i=1$. Indeed, we deduce from the symmetry of $(a_{ij})$ that
\begin{align*}
  \frac{dE_R}{dt} &= \sum_{i,j=1}^N\int_{\R^3}a_{ij}\rho_j\pa_t\rho_i dx
  = -\sum_{i,j=1}^N\int_{\R^3}a_{ij}\na_x\rho_j\cdot
  \bigg(\frac{\sigma}{m_i\nu_i}\na_x\rho_i + \frac{\rho_i}{\nu_i}
  \sum_{k=1}^N a_{ik}\na_x\rho_k\bigg)dx \\
  &= -\sigma\sum_{i,j=1}^N\int_{\R^3}\frac{a_{ij}}{m_i\nu_i}\na_x\rho_i
  \cdot\na_x\rho_j dx - \sum_{i=1}^n\int_{\R^3}\frac{\rho_i}{\nu_i}
  \bigg(\sum_{j=1}^N a_{ij}\na_x\rho_j\bigg)^2 dx,
\end{align*}
and the right-hand side is nonpositive since $(a_{ij})$ is positive definite and $\nu_im_i=1$. The Rao entropy can be interpreted as the limiting potential energy 
\begin{align*}
  E_{\rm pot}^\eps &= -\sum_{i=1}^N 
  \frac{m_i}{2}\int_{\R^3}\int_{\R^3}\phi_i^\eps f_i^\eps d\xi dx
\end{align*}
associated to the kinetic model (still in the isothermal setting):
\begin{align*}
  \lim_{\eps\to 0} E_{\rm pot}^\eps
  &= -\sum_{i=1}^N\frac{m_i}{2}\int_{\R^3}\int_{\R^3}\phi_i M_i(f_i^0)
  d\xi dx = -\frac12\sum_{i=1}^N\int_{\R^3}\phi_i\rho_i dx \\
  &= \frac12\sum_{i,j=1}^N\int_{\R^3}a_{ij}\rho_j\rho_i = E_R.
\end{align*}
It is an open problem whether the Rao entropy can be interpreted as a potential energy in more general kinetic settings. 
\qed\end{remark}


\begin{appendix}
\section{Computation of $v_{ij}$ and $\theta_{ij}$}\label{sec.theta}

\begin{lemma}\label{lem.thetaij}
Let the collision operator \eqref{bgk.Qij} satisfy the 
invariance properties
\begin{align}
  \int_{\R^3}Q_{ij}(f_i^\eps,f_j^\eps)d\xi &= 0, \nonumber \\
  \int_{\R^3}Q_{ij}(f_i^\eps,f_j^\eps)m_i\xi d\xi 
  + \int_{\R^3}Q_{ji}(f_j^\eps,f_i^\eps)m_j\xi d\xi &= 0, \nonumber \\
  \int_{\R^3}Q_{ij}(f_i^\eps,f_j^\eps)m_i|\xi|^2 d\xi 
  + \int_{\R^3}Q_{ji}(f_j^\eps,f_i^\eps)m_j|\xi|^2 d\xi &= 0. \nonumber
\end{align}
Assume that $v_{ij}^\eps=v_{ji}^\eps$ and $\theta_{ij}^\eps=\theta_{ji}^\eps$. Then $v_{ij}^\eps = \alpha_{ij}^\eps v_i^\eps + \alpha_{ji}^\eps v_j^\eps$ and
\begin{align*}
  \theta_{ij}^\eps &= (\beta_{ij}^\eps \theta_i^\eps
  + \beta_{ji}^\eps\theta_j^\eps)
  + \frac{\eps^2}{3}\big(\beta_{ij}^\eps m_i
  (|v_i^\eps|^2-|v_{ij}^\eps|^2)
  + \beta_{ji}^\eps m_j(|v_j^\eps|^2-|v_{ij}^\eps|^2)\big) \\
  &= (\beta_{ij}^\eps\theta_i^\eps + \beta_{ji}^\eps\theta_j^\eps)
  + \frac{\eps^2}{3}\alpha_{ij}^\eps\alpha_{ji}^\eps
  \big(\beta_{ij}^\eps m_i+\beta_{ji}^\eps m_j\big)
  |v_i^\eps-v_j^\eps|^2, 
\end{align*}
where the values $\alpha_{ij}^\eps$ and $\beta_{ij}^\eps$ are defined in \eqref{bgk.alpha}.
\end{lemma}

\begin{proof}
We insert the definition of $Q_{ij}(f_i^\eps,f_j^\eps)$ into the momentum conservation equation and use the moments \eqref{bgk.MMeps} as well as $v_{ij}^\eps=v_{ji}^\eps$:
\begin{align*}
  0 &= \int_{\R^3}\nu_{ij}(M_{ij}-f_i^\eps)m_i\xi d\xi
  + \int_{\R^3}\nu_{ji}(M_{ji}-f_j^\eps)m_j\xi d\xi \\
  &= \eps\nu_{ij}(\rho_i^\eps v_{ij}^\eps - \rho_i^\eps v_i^\eps)
  + \eps\nu_{ji}(\rho_j^\eps v_{ij}^\eps - \rho_j^\eps v_j^\eps).
\end{align*}
Solving this equation for $v_{ij}^\eps$ yields $v_{ij}^\eps = \alpha_{ij}^\eps v_i^\eps + \alpha_{ji}^\eps v_j^\eps$. Next, we consider the energy conservation equation and take into account $\theta_{ij}^\eps=\theta_{ji}^\eps$:
\begin{align*}
  0 &= \int_{\R^3}\nu_{ij}(M_{ij}-f_i^\eps)m_i|\xi|^2 d\xi
  + \int_{\R^3}\nu_{ji}(M_{ji}-f_j^\eps)m_j|\xi|^2 d\xi \\
  &= 3\nu_{ij}n_i^\eps(\theta_{ij}^\eps-\theta_i^\eps) 
  + \eps^2\nu_{ij}m_in_i^\eps(|v_{ij}^\eps|^2-|v_i|^2) \\
  &\phantom{xx}+ 3\nu_{ji}n_j^\eps(\theta_{ij}^\eps-\theta_j^\eps) 
  + \eps^2\nu_{ji}m_jn_j^\eps(|v_{ij}^\eps|^2-|v_j|^2).
\end{align*}
We solve for $\theta_{ij}^\eps$, giving
\begin{align*}
  \theta_{ij}^\eps &= \frac{\nu_{ij} n_i^\eps\theta_i^\eps
  + \nu_{ji}n_j^\eps\theta_j^\eps}{\nu_{ij} n_i^\eps+\nu_{ji}n_j^\eps}
  + \frac{\eps^2}{3}\frac{\nu_{ij}m_in_i^\eps(|v_i^\eps|^2-|v_{ij}^\eps|)
  + \nu_{ji}m_jn_j^\eps(|v_j^\eps|^2-|v_{ij}^\eps|)
  }{\nu_{ij} n_i^\eps+\nu_{ji}n_j^\eps} \\
  &= (\beta_{ij}^\eps\theta_i^\eps + \beta_{ji}^\eps\theta_j^\eps)
  + \frac{\eps^2}{3}\big(\beta_{ij}^\eps m_i(|v_i^\eps|^2-|v_{ij}^\eps|)
  + \beta_{ji}^\eps m_j(|v_j^\eps|^2-|v_{ij}^\eps|)\big).
\end{align*}
The $\eps^2$-term can be reformulated by using $\alpha_{ij}^\eps+\alpha_{ji}^\eps=1$ and $|v_{ij}^\eps|^2=(\alpha_{ij}^\eps)^2|v_i^\eps|^2 + (\alpha_{ji}^\eps)^2|v_j^\eps|^2 + 2\alpha_{ij}^\eps\alpha_{ji}^\eps v_i^\eps\cdot v_j^\eps$:
\begin{align*}
  \beta_{ij}^\eps m_i&(|v_i^\eps|^2-|v_{ij}^\eps|)
  + \beta_{ji}^\eps m_j(|v_j^\eps|^2-|v_{ij}^\eps|) \\
  &= \frac{\nu_{ij}m_in_i^\eps+\nu_{ji}m_jn_j^\eps}{
  \nu_{ij}n_i^\eps+\nu_{ji}n_j^\eps}
  \big(\alpha_{ij}^\eps(|v_i^\eps|^2-|v_{ij}^\eps|)
  + \alpha_{ji}^\eps(|v_j^\eps|^2-|v_{ij}^\eps|)\big) \\
  &= (\beta_{ij}^\eps m_i+\beta_{ji}^\eps m_j)
  \big(\alpha_{ij}^\eps|v_i^\eps|^2
  + \alpha_{ji}^\eps|v_j^\eps|^2 - |v_{ij}^\eps|^2\big) \\
  &= (\beta_{ij}^\eps m_i+\beta_{ji}^\eps m_j)
  \big(\alpha_{ij}^\eps(1-\alpha_{ij}^\eps)
  |v_i^\eps|^2 + \alpha_{ji}^\eps(1-\alpha_{ji}^\eps)|v_j^\eps|^2
  - 2\alpha_{ij}^\eps\alpha_{ji}^\eps v_i^\eps\cdot v_j^\eps\big) \\
  &= (\beta_{ij}^\eps m_i+\beta_{ji}^\eps m_j)
  \alpha_{ij}^\eps\alpha_{ji}^\eps|v_i^\eps-v_j^\eps|^2,
\end{align*}
finishing the proof.
\end{proof}


\section{Scaling}\label{sec.scale}

We wish to interpret the scaling $\sigma=\sqrt{\eps}$ in Section \ref{sec.hf} by writing the kinetic equation
\begin{align}\label{be}
  \pa_t f_i + \xi\cdot\na_x f_i + \frac{F_i}{m^*}\cdot\na_\xi f_i
  = Q_i(f), \quad i=1,\ldots,N,
\end{align}
in dimensionless form. Here, $F_i$ denotes a force and $m^*$ the associated particle mass. We choose a characteristic time $\tau$, a characteristic length $\lambda$, and the characteristic velocity $u_0=\sqrt{k_B\theta/m^*}$ (proportional to the sound speed of an ideal gas), where $k_B$ is the Boltzmann constant and $\theta$ the ensemble temperature. The characteristic time and length define a second velocity scale, $u=\lambda/\tau$. The value $\lambda_0=u_0\tau_0$ (mean-free path) is the distance that a particle travels between two consecutive collisions in time $\tau_0$. This defines the dimensionless variables
\begin{align*}
  t = \tau t_s, \quad x = \lambda x_s, \quad \xi = u_0 \xi_s,
\end{align*}
as well as the dimensionless force $F_i=F_0 F_{i,s}$, where $F_0$ will be determined below, and the dimensionless collision operator $Q_{i}(f)=\tau_0^{-1}Q_{i,s}(f)$. Then \eqref{be} becomes, after omitting the index ``$s$'' for the dimensionless variables and functions,
\begin{align*}
  \mathrm{Ma}\,\pa_t f_i + \xi\cdot\na_x f_i 
  + \frac{F_0\lambda}{m^*u_0^2}F_i\cdot\na_\xi f_i 
  = \frac{1}{\mathrm{Kn}}Q_i(f),
\end{align*}
where the Mach number Ma and Knudsen number Kn are defined by
\begin{align*}
  \mathrm{Ma} = \frac{u}{u_0}, \quad 
  \mathrm{Kn} = \frac{\lambda_0}{\lambda}.
\end{align*}

The diffusive scaling in Section \ref{sec.bgkmodel} is obtained by assuming that the Mach and Knudsen numbers are equal, $\mathrm{Mu}=\mathrm{Kn}=\eps$. This means that the mean free path is small compared to the charatceristic length, i.e., collisions are very frequent, and the time $\tau_0$ between two consecutive collisions is very small compared to the characteristic time $\tau$, $\tau_0/\tau=\mathrm{Ma}\cdot\mathrm{Kn}=\eps^2$. Furthermore, the characteristic force $F_0$ is of the same order as the kinetic energy $m^*u_0^2/2$ of a particle with mass $m^*$ and moving with the sound speed, $F_0=m^*u_0^2$. This leads to
\begin{align*}
  \eps \pa_t f_i + \xi\cdot\na_x f_i + F\cdot\na_\xi f_i 
  = \frac{1}{\eps}Q_i(f).
\end{align*}
For the scaling $\sigma=\sqrt{\eps}$, we suppose that the Knudsen number is much smaller than the Mach number, $\mathrm{Ma}=\eps$ and $\mathrm{Kn}=\eps^2$. This means that $\tau_0/\tau=\eps^3$, i.e., collisions are even more frequent than in the diffusive scale. Furthermore, the force is dominant at the scale $1/\eps$, and we choose $F_0=m^*u_0^2/\eps$. We find that
\begin{align*}
  \eps\pa_t f_i + \xi\cdot\na_x f_i + \frac{1}{\eps}F\cdot\na_\xi f_i
  = \frac{1}{\eps^2}Q_i(f),
\end{align*}
which yields \eqref{br.sigma} after replacing $\eps$ by $\sqrt{\eps}$.

When performing the Chapman--Enskog expansion $f_i=M_i+\eps g_i^\eps$ in the BGK model $Q_i(f)=\nu_i(M_i(f_i)-f_i)$, the right-hand side becomes $-g_i/\eps$, while the convective part $\xi\cdot\na_x M_i$ that usually provides the diffusion term is of order one only. The force term contains $F\cdot\na_\xi M_i/\eps$ which balances the collision term and leads in the limit $\eps\to 0$ to the flux
\begin{align*}
  \int_{\R^3}g_i\xi d\xi = -F\cdot\int_{\R^3}\na_\xi M_i\xi d\xi 
  = F\int_{\R^3}M_i d\xi = Fn_i,
\end{align*}
which acts like a convective term. Since in our model $F=\na_x\phi_i$ depends on the gradients of $\rho_j$, we arrive to a diffusion model.

\end{appendix}



\begin{thebibliography}{11}

\bibitem{AAP02} P.~Andries, K.~Aoki, and B.~Perthame. A consistent BGK-type model for gas mixtures. {\em J. Stat. Phys.} 106 (2002), 993--1018.

\bibitem{AnSi22} B.~Anwasia and S.~Simi\'c. Maximum entropy principle approach to a non-isothermal Maxwell--Stefan diffusion model. {\em Appl. Math. Lett.} 129 (2022), no.~107949, 9 pages.

\bibitem{BBTW15} N.~Bellomo, A.~Bellouquid, Y.~Tao, and M.~Winkler. Toward a mathematical theory of Keller--Segel models of pattern formation in biological tissues. {\em Math. Models Meth. Appl. Sci.} 25 (2015), 1663--1763.

\bibitem{BKZ18} M.~Bendahmane, F.~Karami, and M.~Zagour. Kinetic-fluid derivation and mathematical analysis of the
cross-diffusion--Brinkman system. {\em Math. Meth. Appl. Sci.} 41 (2018), 6288--6311.

\bibitem{Ben24} E.~Benilov. Multispecies Bhatnagar--Gross--Krook models and the Onsager reciprocal relations. {\em Phys. Rev. E} 110 (2024), no.~034125, 5 pages.

\bibitem{BGK54} P.~Bhatnagar, E.~Gross, and M.~Krook. A model for collision processes in gases. I. Small amplitude processes in charged and neutral one-component systems. {\em Phys. Rev.} 94 (1954), 511--525.

\bibitem{BBGSP18} A.~Bobylev, M.~Bisi, M.~Groppi, G.~Spiga, and I.~Potapenko. A general consistent BGK model for gas mixtures. {\em Kinetic Relat. Models} 11 (2018), 1377--1393.

\bibitem{BGS15} L.~Boudin, B.~Grec, and F.~Salvarani. The Maxwell--Stefan diffusion limit for a kinetic model of mixtures. {\em Acta Appl. Math.} 136 (2015), 79--90.

\bibitem{BrGr23} M.~Briant and B.~Grec. Rigorous derivation of the Fick cross-diffusion system from the multi-species Boltzmann equation. {\em Asympt. Anal.} 135 (2023), 55--80.

\bibitem{Bri49} H.~Brinkman. A calculation of the viscous force exerted by a flowing fluid on a dense swarm of particles. {\em Appl. Sci. Res.} 1 (1949), 27--34.

\bibitem{BuTr83} S.~Busenberg and C.~Travis. Epidemic models with spatial spread due to population migration. {\em J. Math. Biol.} 16 (1983), 181--198.

\bibitem{Cer69} C.~Cercignani. {\em Mathematical Methods in Kinetic Theory}. Plenum Press, New York, 1969.

\bibitem{CCDJ24} J.~A.~Carrillo, X.~Chen, B.~Du, and A.~J\"ungel. Fluid relaxation approximation of the Busenberg--Travis cross-diffusion system. {\em Commun. Math. Phys.} 406 (2025), no.~151, 29 pages.

\bibitem{CMPS04} F.~Chalub, P.~Markowich, B.~Perthame, and C.~Schmeiser. Kinetic models for chemotaxis and their drift-diffusion limits. {\em Monatsh. Math.} 142 (2004), 123--141.

\bibitem{CDJ19} L.~Chen, E.~Daus, and A.~J\"ungel. Rigorous mean-field limit and cross-diffusion. {\em Z. Angew. Math. Phys.} 70 (2019), no.~122, 21 pages.

\bibitem{ChJu04} L.~Chen and A.~J\"ungel. Analysis of a multi-dimensional parabolic population model with strong
cross-diffusion. {\em SIAM J. Math. Anal.} 36 (2004), 301--322.

\bibitem{CFI25} Y.-P.~Choi, S.~Fagioli, and V.~Iorio. Small inertia limit for coupled kinetic swarming models. {\em J. Nonlin. Sci.} 35 (2025), no.~39, 52 pages.

\bibitem{GSB89} V.~Garz\'o, A.~Santos, and J.~Brey. A kinetic model for a
multicomponent gas. {\em Phys. Fluids A} 1 (1989), 380--383.

\bibitem{GeJu24} S.~Georgiadis and A.~J\"ungel. Global existence of weak solutions and weak--strong uniqueness for nonisothermal Maxwell--Stefan systems. {\em Nonlinearity} 37 (2024), no.~075016, 33 pages.

\bibitem{GrSi23} B.~Grec and S.~Simi\'c. Higher-order Maxwell--Stefan model of diffusion. {\em La Matematica} 2 (2023), 962--991.

\bibitem{GrKr56} E.~Gross and M.~Krook. Model for collision processes in gases: Small-amplitude oscillations of charged two-component systems. {\em Phys. Rev.} 102 (1956), 593--604.

\bibitem{HHM17} J.~Haack, C.~Hauck, and M.~Murillo. A conservative, entropic multispecies BGK model. {\em J. Stat. Phys.} 168 (2017), 826--856.

\bibitem{HeJu21} C.~Helmer and A.~J\"ungel. Analysis of Maxwell--Stefan systems for heat conducting fluid mixtures. {\em Nonlin. Anal. Real World Appl.} 59 (2021), no.~103263, 19 pages.

\bibitem{HuSa17} H.~Hutridurga and F.~Salvarani. Maxwell--Stefan diffusion asymptotics for gas mixtures in non-isothermal setting. {\em Nonlin. Anal.} 159 (2017), 285--297.

\bibitem{HuSa18} H.~Hutridurga and F.~Salvarani. Existence and uniqueness analysis of a non-isothermal cross-diffusion system of Maxwell--Stefan type. {\em Appl. Math. Lett.} 75 (2018), 108--113.

\bibitem{JuSt13} A.~J\"ungel and I.~Stelzer. Existence analysis of Maxwell--Stefan systems for multicomponent mixtures. {\em SIAM J. Math. Anal.} 45 (2013), 2421--2440.


\bibitem{KPP17} C.~Klingenberg, M.~Pirner, and G.~Puppo. A consistent kinetic model for a two-component mixture with an application to plasma. {\em Kinetic Relat. Models} 10 (2017), 445--465.

\bibitem{MST25} G.~Martal\`o, A.~Soares, and R.~Travaglini. A BGK-type model for multi-component gas mixtures undergoing a bimolecular chemical reaction. {\em J. Stat. Phys.} 192 (2025), no.~3, 29 pages.

\bibitem{MaTr24} G.~Martal\`o and R.~Travaglini. A reaction--cross-diffusion model derived from kinetic equations for gas
mixtures. {\em Physica D} 459 (2024), no.~134029, 10 pages.

\bibitem{Max66} J.~C.~Maxwell. On the dynamical theory of gases. {\em  Philos. Trans. Roy. Soc. London} 157 (1866), 49--88.

\bibitem{Pla19} R.~Plaza. Derivation of a bacterial nutrient-taxis system with doubly degenerate cross-diffusion as the parabolic limit of a
velocity-jump process. {\em J. Math. Biol.} (2019) 78, 1681--1711.

\bibitem{Pup19} G.~Puppo. Kinetic models of BGK type and their numerical integration. {\em Riv. Mat. Univ. Parma} 10 (2019), 299-349.

\bibitem{Ste71} J.~Stefan. \" Uber das Gleichgewicht und Bewegung, insbesondere die Diffusion von Gasgemengen. {\em Sitzungsberichte Kaiserl. Akad. Wiss. Wien} 63 (1871), 63--124.

\bibitem{TaJu08} S.~Taguchi and A.~J\"ungel. A two-surface problem of the electron flow in a semiconductor on the basis of kinetic theory. {\em J. Stat. Phys.} 130 (2008), 313--342.

\bibitem{ZaJu17} N.~Zamponi and A.~J\"ungel. Analysis of degenerate cross-diffusion population models with volume filling. {\em Ann. Inst. H. Poincar\'e Anal Non Lin.} 4 (2017), 1--29.

\end{thebibliography}
\end{document}